\documentclass{article}

\usepackage[utf8]{inputenc} 
\usepackage[T1]{fontenc}    
\usepackage{hyperref}       
\usepackage{url}            
\usepackage{booktabs}       
\usepackage{amsfonts}       
\usepackage{nicefrac}       
\usepackage{microtype}      
\usepackage{lipsum}
\usepackage{amsmath}
\usepackage[normalem]{ulem}

\usepackage{amsfonts,amssymb,mathtools}
\usepackage{graphicx}
\usepackage{amsthm}
\usepackage{xcolor}

\newcommand{\N}{\mathbb{N}}
\newcommand{\R}{\mathbb{R}}
\newcommand{\s}{\mathbb{S}}
\newcommand{\e}{\mathbb{E}^3}

\newtheorem{theorem}{Theorem}[section]

\newtheorem{lemma}{Lemma}[section]

\newtheorem{definition}{Definition}[section]
\newtheorem{remark}{Remark}[section]
\newtheorem{example}{Example}[section]

\usepackage{cite}
\usepackage{authblk}

\newcommand{\p}{\partial}
\newcommand{\bb}{\begin{equation}}
\newcommand{\ee}{\end{equation}}
\newcommand{\ba}{\begin{array}}
\newcommand{\ea}{\end{array}}
\newcommand{\f}{\frac}
\usepackage[all]{xy}
\newcommand{\ds}{\displaystyle}

\newcommand{\al}{\alpha}
\newcommand{\be}{\beta}

\newcommand{\M}{\cal M}
\newcommand{\Span}{\text{Span}}
\newcommand{\C}{C_{\text{per}}}
\newcommand{\Gr}{\text{Gr}}
\numberwithin{equation}{section}
\usepackage{csquotes}

\title{Existence and uniqueness of periodic pseudospherical surfaces emanating from Cauchy problems}

\author[1]{Nilay Duruk Mutlubas}  \author[2,3]{Igor Leite Freire}
\affil[1]{Faculty of Engineering and Natural Sciences, Sabanci University, Turkey
\texttt{nilay.duruk@sabanciuniv.edu}}

\affil[2]{Department of Mathematical Sciences,
Loughborough University\\
 LE11 3TU Epinal Way, Loughborough, United Kingdom\\
\texttt{I.Leite-Freire@lboro.ac.uk}
}
\affil[3]{Departamento de Matemática, Universidade Federal de São Carlos\\
Rodovia Washington Luís, Km 235, 13565-905\\
São Carlos-SP, Brasil\\
  \texttt{igor.freire@ufscar.br} \\
  \texttt{igor.leite.freire@gmail.com}}

\begin{document}
\maketitle
\begin{abstract}
\centering\begin{minipage}{\dimexpr\paperwidth-9cm}
\textbf{Abstract:} We study implications and consequences of well-posed solutions of Cauchy problems of a Novikov equation describing pseudospherical surfaces. We show that if the co-frame of dual one-forms satisfies certain conditions for a given periodic initial datum, then there exists exactly two families of periodic one-forms satisfying the structural equations for a surface. Each pair then defines a metric of constant Gaussian curvature and a corresponding Levi-Civita connection form. We prove the existence of universal connection forms giving rise to second fundamental forms compatible with the metric. The main tool to prove our geometrical results is the Kato's semi-group approach, which is used to establish well-posedness of solutions of the Cauchy problem involved and ensure $C^1$ regularity for the first fundamental form and the Levi-Civita connection form.

\vspace{0.1cm}
\textbf{2020 AMS Mathematics Classification numbers}: 35B10, 53A05, 58J60, 35A30.


\end{minipage}
\end{abstract}
\bigskip
\newpage
\tableofcontents
\newpage

\section{Introduction}\label{sec1}

In \cite{sasaki} Sasaki made a remarkable observation, showing that solutions of integrable equations solved by the AKNS $2\times2$ method \cite{akns} give rise to metrics of pseudospherical surfaces with Gaussian curvature ${\cal K}=-1$, see \cite[section 2]{sasaki} and \cite[section 1]{chern} for further details.

Later on, Chern and Tenenblat, in their seminal paper \cite{chern}, introduced the notion of equations describing pseudospherical surfaces (PSS equation) and gave a systematic way for finding them.

The works by Sasaki \cite{sasaki} and Chern and Tenenblat \cite{chern} showed a deep connection between integrability and differential geometry of pseudospherical surfaces, unsurprisingly, leading to a new notion of integrability, see \cite[Definition 2]{reyes2000} and \cite[page 245]{reyes2006-sel}.

Notwithstanding its relevance in terms of integrability, the work by Chern and Tenenblat made an in-depth investigation on certain very peculiar equations having the following property: with some exceptions (that we will discuss later), their solutions give rise to metrics with constant Gaussian curvature. This fact {\it per se} has been known for a long time for the Sine-Gordon equation, see \cite[Section 1]{rogers}, but its systematic study, implications and applications to other equations, potential links with integrable systems and construction of conserved quantities made \cite{chern} a paramount work.

The metric and the Gaussian curvature are intrinsic properties of a surface, but alas insufficient to completely describe it. To this end, we need further information provided by its second fundamental form. While the first fundamental form (metric) can be though as the {\it way a two-dimensional being walks} on the surface (an intrinsic aspect), the second fundamental form tells us {\it how the surface behaves from, or looks like to, an observer located outside it}.

Given the importance of the second fundamental form and the relevance of the observations and ideas introduced in \cite{sasaki} and \cite{chern}, respectively, it is somewhat surprising that it was taken nearly three decades from \cite{chern} until the first works \cite{kah-book,kah-cag,kah} concerning second fundamental forms of the surfaces defined by the solutions of PSS equations.

Although the work by Chern and Tenenblat was born in the context of integrability of differential equations, and most of the follow-up works, not to say all, were concerned with these connections, see \cite{cat,reyes2000,reyes2002,reyes2006-sel,reyes2006-jde,reyes2011}, along time the integrability aspects were put aside and the research carried out has been more focused on geometric aspects and classification of equations describing PSS, see \cite{tarcisio,ding,kah-book,kah-cag,kah,keti2015,tito} and references therein. 

A considerable number of relevant PSS equations can be seen as dynamical systems in certain Banach spaces, and from the point of view of analysis of PDEs, qualitative aspects of their solutions are obtained from Cauchy problems, meaning that not only the equation is relevant, but also a condition satisfied by a given solution of the equation at a given time, very often $t=0$ (initial condition or datum). Usually, the regularity of the initial datum determines that of the corresponding solution of the equation.

From a geometric perspective, solutions emanating from Cauchy problems involving an equation for an unknown $u=u(x,t)$ can be seen as follows: given a certain curve $x\mapsto(x,0,u_0(x))$, can we find a solution $u$ for the equation such that the given curve belongs to the graph of $u$? Moreover, what does the regularity of the curve say about the graph of $u$? Is $u$ the only solution of the equation whose graph contains the given curve? 

Despite being a topic mostly concerned with analysis of PDEs, the paragraph above shows that the problem of existence and uniqueness of solutions (well-posedness) of PDEs makes sense in the context of PSS equations. Surprisingly, it seems this topic has been out of the agenda of the literature of PSS equations. The purpose of our paper is to shed light on it.

The main motivation for us to undertake the research reported in the present work are recent results reported in \cite{nazime} concerned with the equation
\bb\label{1.0.1}
u_t-u_{txx}=\p_x(2-\p_x)(1+\p_x)u^2,
\ee
which was discovered in \cite{nov} and recently has been proved to be geometrically integrable \cite{freire-tito-sam}, meaning that its solutions  describe a non-trivial family of pseudospherical surfaces. 

Equation \eqref{1.0.1} was studied in \cite{guo-bvp, li-na, li-jmaa,liu-jde} from the point of view of qualitative analysis, such as existence and uniqueness of solutions. More recently, in \cite[Theorem 5.1]{nazime} results from \cite{li-na,liu-jde} were combined with \cite{freire-tito-sam} to prove the existence of $C^\omega$ (metrics for) pseudospherical surfaces arising from the solutions of \eqref{1.0.1}.

The aforementioned result proved in \cite{nazime}, despite being established for $C^\omega$ solutions, strongly indicates the possibility of relating Cauchy problems and pseudospherical surfaces. Actually, it made such a connection, but by considering solutions emanating from an initial datum with $C^\omega$ regularity. The question is: Could we consider the same problem replacing a $C^\omega$ initial datum by one with lower regularity? Can we consider a periodic initial datum?

In line with the comments above, the vast majority of works in the field of PSS equations considers explicit or implicitly $C^\infty$ solutions of the PSS equations, which technically avoid problems regarding regularity (that is, how much smooth the object is) and lead to $C^\infty$ metrics. A simple question then arises: What may happen if we consider solutions with regularity other than $C^\infty$?

The answer to the questions above is given in our first result.

\begin{theorem}\label{thm1.1}
Let $u_0\in H^4(\s)$ be a non-trivial and non-constant initial datum, with $u-u_0''>0$, and consider the Cauchy problem
\bb\label{1.0.2}
\left\{
\ba{l}
u_t-u_{txx}=\p_x(2-\p_x)(1+\p_x)u^2,~~~~~~x\in \mathbb{R},~~~~~t>0,\\
\\
u(x,0)=u_0(x),~~~~~~x\in \mathbb{R},\\
\\
u(x,t)=u(x+1,t),~~~~~~x\in \mathbb{R},~~~~~t>0.
\ea\right.
\ee

Then there exists triads of $C^1$ one-forms $\omega_1,\,\omega_2,\,\omega_3$, with
\bb\label{1.0.3}
\omega_i=f_{i1}dx+f_{i2}dt,\quad 1\leq i\leq 3,
\ee
\bb\label{1.0.4}
f_{p1}=\mu_pf_{11}+\eta_p,\quad 1\leq p\leq 2,
\ee
where $\mu_p,\,\eta_p\in\R$, such that the forms \eqref{1.0.3} are defined on $U=\R\times(0,\infty)$, periodic with respect to $x$, and define a PSS whenever $\nabla u\neq(0,0)$, with $\omega_3$ being the Levi-Civita connection of the metric determined by $\omega_1$ and $\omega_2$. 

Moreover, fixed a pair $\{\omega_1,\,\omega_2\}$ and $p\in U$, there exists connection forms $\omega_{13}=a\omega_1+b\omega_2$, $\omega_{23}=b\omega_1+c\omega_2$, where $a,\,b,\,c$ are $C^\infty$ functions defined on an open neighborhood $V\subseteq U$ of $p$, such that $\{\omega_1,\omega_2,\omega_{13},\,\omega_{23}\}$ defines a PSS of Gaussian curvature ${\cal K}=-1$.
\end{theorem}

In section \ref{sec2} we shall present all pertinent definitions and notions, but for now it suffices to say that a function belonging to $H^4(\s)$ is a real valued periodic function, with period $1$, of class $C^3$. 

Our theorem \ref{thm1.1} can be seen as an existence and uniqueness result for PSS surfaces. In fact, it says that from solutions of equation \eqref{1.0.1} whose graphs contain the regular curve $x\mapsto(x,0,u_0(x))$, with $u_0\in H^4(\s)$, we can obtain an open set $V\subseteq\R^2$ in which we have two possible choices to define a first fundamental form for a PSS surface with Gaussian curvature ${\cal K}=-1$. Moreover, we can locally define connection forms on each point of $V$. This fact, jointly with Bonnet theorem, tells us that we can locally define a PSS surface embedded in $\R^3$. 

A key point to understand and prove theorem \ref{thm1.1} is determining whether the problem \eqref{1.0.2} is well-posed. To this end, recognising the presence of the Helmholtz operator $\Lambda^{2}=1-\p_x^2$ in \eqref{1.0.1}, we can rewrite the problem \eqref{1.0.2} in an alternative form, given by
\bb\label{1.0.5}
\left\{
\ba{l}
u_t-2uu_x=\p_x\Lambda^{-2}(u^2+(u^2)_x),~~~~~~x\in \mathbb{R},~~~~~t>0,\\
\\
u(x,0)=u_0(x),~~~~~~x\in \mathbb{R},\\
\\
u(x,t)=u(x+1,t),~~~~~~x\in \mathbb{R},~~~~~t>0.
\ea\right.
\ee

\begin{theorem}\label{thm1.2}

Let $u_0\in H^{s}(\mathbb{S})$, $s>3/2$ be a given initial datum. Then there exists a maximal time of existence $T>0$, depending on $u_0$, such that there is a unique solution $u$ to \eqref{1.0.5} satisfying
$    u\in C^0(H^{s}(\mathbb{S}),[0,T))\cap C^1(H^{s-1}(\mathbb{S}),[0,T))$.
Moreover, the map $u_0\in H^{s}(\mathbb{S})	\rightarrow u$, is continuous from $H^{s}(\mathbb{S})$ to $C^0(H^{s}(\mathbb{S}),[0,T))\cap C^1(H^{s-1}(\mathbb{S}),[0,T))$ and $T$ is independent of $s$.
\end{theorem}

The inverse of the Helmholtz operator, denoted by $\Lambda^{-2}$ and acting on a function $f$, is defined by the convolution $g\ast f$, where 
\bb\label{1.0.6}
g(x)=\f{\cosh{(x-\lfloor x\rfloor-1/2)}}{2\sinh{(1/2)}}
\ee
and $\lfloor \cdot\rfloor$ denotes the greatest integer function. 

Some previous results in the literature had already proved well-posedness results concerning periodic solutions of the problem \eqref{1.0.5}, see \cite{liu-jde}, but it was either shown that $u\in C^0([0,T),H^{s}(\mathbb{S}))$ or $u$ is $C^\omega$ in both variables on a certain domain, see \cite[Theorem 1.1]{liu-jde} and \cite[Theorem 1.4]{liu-jde}, respectively. Although these results show the existence and uniqueness of solutions for a very large class of functions, they are unsuitable for our purposes because we need solutions with $C^1$ regularity with respect to $t$. 

The importance of our theorem \ref{thm1.2} comes from just the fact that it ensures we have $C^1$ solutions in both variables. Actually, this is a consequence of the Sobolev Lemma. In particular, it tells us that the solutions granted by theorem \ref{thm1.2} are strong solutions for the (non-local) first order PDE in \eqref{1.0.5}. A natural question then arises: is a strong solution of the equation in \eqref{1.0.5} also a strong solution of \eqref{1.0.1}? In general the answer is no! However, requiring enough regularity of the initial datum we can find solutions for the Cauchy problem \eqref{1.0.5} that are also strong solutions for \eqref{1.0.1}, and therefore, simultaneously strong solutions for both formulations of the equation.

In fact, whenever we consider an initial datum in $H^4(\s)$, the corresponding solution provided by theorem \ref{thm1.2} not only is a $C^1$ solution (in both variables), but the Sobolev Lemma also implies that $x\mapsto u(x,t)$, $t$ fixed, is $C^3$, meaning that the solution emanating from \eqref{1.0.5} is a strong, or classical, solution for \eqref{1.0.1}, which makes sense to be considered in the study of PSS and differential equations. 

Although the regularity of the initial datum is enough to make the corresponding solution of \eqref{1.0.5} a strong solution of \eqref{1.0.1}, it is insufficient to guarantee that the solution is global, in the sense that it is defined for every $t>0$. We can have global solutions requiring little more from the initial datum.

\begin{theorem}\label{thm1.3}
If $u_0\in H^4(\s)$ is a non-trivial initial datum, and $u_0(x)-u_0''(x)>0$, $x\in\R$, then the solution of the problem \eqref{1.0.5} exists for any $t>0$. Moreover, $u\in C^1(\R\times(0,\infty))$ and $x\mapsto u(x,t)$ is a $C^3$ periodic function, for each fixed $t>0$.
\end{theorem}

It is important to note that in view of the Sobolev Lemma, $H^4(\s)$ is continuously and densely embedded in $H^s(\s)$, for $s\in(3/2,4)$. Therefore, both theorems \ref{thm1.2} and \ref{thm1.3} tell us that the problem \eqref{1.0.2} has only one solution $u$. 

\subsection{Novelty and challenges of the manuscript} 

We study the problem of PSS determined by the solutions of a given equation from the perspective of well-posedness of Cauchy problems, which as far as we know, has not been considered yet.

Let us highlight the relevance of our results by discussing the following problem: suppose we know a curve from the graph of a (unknown) solution of a PSS equation. Can we precisely describe the corresponding PSS?

Let us exemplify by considering the function $u_c(x,t):=e^{x-ct}$. Any member of the family ${\cal U}=\{u_c,\,c\in\R\}$ is a solution of \eqref{1.0.1} defined for $(x,t)\in\R\times(0,\infty)$, see \cite[page 5]{nazime}. For any $u\in{\cal U}$, let
$\Gr(u)=\{(x,t,u(x,t)),\,\,x\in\R,\,\,t>0\}.$

Consider the curve $\Gamma$, given by $x\mapsto(x,0,e^x)$ and let $\p \Gr(u)$ denote the boundary of a set $\Gr(u)$. Then it is easy to see that $\Gamma\subseteq \p\Gr(u)$, for any $u\in{\cal U}$. In particular, we have
$$\Gamma\subseteq \bigcap_{u\in{\cal U}} \p\Gr(u).$$

If we allow $t=0$ in the domain of $u_c$, then the curve $\Gamma$ belongs to the corresponding graph, but for our purposes it is enough to consider it lying in the boundary, the latter being disjoint from the graph.

On the other hand, \eqref{1.0.1} is a PSS equation (see \cite[Theorem 1]{freire-tito-sam}), and for any member $u_c$ of ${\cal U}$ we can construct a PSS $U_c$ in an intrinsic way (note that these solutions satisfy all required conditions for the existence of a PSS, see the comments after Theorem 1 in \cite{freire-tito-sam}).

Due to the fact that the curve $\Gamma$ belongs to the boundary of any graph of the solutions $u_c$, we cannot determine any specific PSS surface only knowing $\Gamma$. 

Our theorem \ref{thm1.1} gives a rather different answer to the same question. In geometric terms, it says that given a curve $\Gamma$ of the form $x\mapsto(x,0,u_0(x))$, as long as $u_0\in H^4(\s)$, we can precisely and intrinsically describe a PSS among all infinite surfaces emanating from all solutions of the PDE \eqref{1.0.1}.

To address this problem we make use of techniques of existence and uniqueness of solutions for PDEs that can be seen as dynamical systems in certain Banach spaces.
In view of this approach, we deal with solutions of the equation that are less regular than those usually considered in the literature of the PSS equations. One of the difficulties to be overcome is concerned with the regularity of the one-forms $\omega_1$ and $\omega_2$ defining the metric of the corresponding PSS. Most of the books in differential geometry require $C^\infty$ forms, although some of them require at least $C^2$ regularity. As we will better discuss in the next section, for the classical theory of curves and surfaces, we can have PSS surfaces from $C^1$ forms satisfying the structure equations for a surface. 

Last but not least, one of the challenges of this paper is that its main result is geometric, but the tools for proving it comes from modern approaches to prove qualitative aspects of solutions of Cauchy problems. For this reason, we tried our best to make clear and explain the technical aspects of each area, so that the readers can have a better reading and appreciation of our work. 

\subsection{Outline of the manuscript}
Since this is a paper focusing on Analysis and Geometry, in the next section we provide an overview about PSS and functional analysis. We also fix the notation, present essential concepts and revisit Kato's semi-group approach, which is the main tool for proving theorems \ref{thm1.2} and \ref{thm1.3}, whose demonstrations are given in sections \ref{sec3} and \ref{sec4}, respectively. Theorem \ref{thm1.1} is proved in section \ref{sec5}, whereas our conclusions are given in section \ref{sec6}.

\section{Notation, notions and preliminaries}\label{sec2}

In this section we introduce and fix the notation used throughout the manuscript. Given its plural and diverse aspects, we also present basic facts and concepts from differential geometry of surfaces and functional analysis, which are the main pillars of the work. Most of the geometric content can be better explored in \cite[Chapter 5]{ilka}, \cite[Chapter 4]{cle} and \cite[Chapter 2]{keti-book}, whereas our main references for functional analysis are \cite[Chapter 3]{iorio} and \cite[Chapter 4]{taylor}.
\subsection{Notation}

Given a function $u=u(x,t)$, by $u(x,\cdot)$ we mean the function $t\mapsto u(x,t)$, for fixed $x$, whereas $u(\cdot,t)$ denotes the function $x\mapsto u(x,t)$, for fixed $t$.

Let $I,J$ two open, non-empty subsets of $\R$. We say that $u\in C^0(I\times J)$ if $u=u(x,t)$ is continuous with respect to both variables $(x,t)\in I\times J$. Partial derivative of $u$ with respect to its first argument will be denoted by $u_x$ or $\p_x u$, whereas $u_t$ or $\p_t u$ will denote partial derivative with respect to the second argument. Higher order derivatives can be considered using the standard conventions.

For a positive integer $k$, we say that $u\in C^k(I\times J)$ if all partial derivatives of $u$ up to order $k$ (including the mixed ones) are continuous. Given a positive integer $n$, we denote by $u_{(n)}$ the set of ordered $n-th$ derivatives of $u$. Also, we say that $u$ is $C^k$ whenever all of its partial derivatives up to order $k$ are continuous on the domain of $u$.

By $C^{3,1}(I\times J)$ we mean the set of function $u:I\times J\rightarrow\R$ such that $u$, $u_x$, $u_t$, $u_{xx}$, $u_{xt}$, $u_{xxx}$ and $u_{xxt}$ belong to $C^0(I\times J)$.

Let $X$ be a Banach space of real valued functions and $I\subseteq\R$. The set $C^0(I,X)$ denotes collection of continuous functions such that $u(t,\cdot)\in X$. More generally, given a positive integer $n$, we say that $u\in C^n(I,X)$ if $\p_t^ku(t,\cdot) \in C^0(I,X)$, $0\leq k\leq n$.

\subsection{Structure equations and pseudospherical surfaces}

Let $\langle\cdot,\cdot\rangle$ be the usual inner product in $\R^3$ and denote the pair $(\R^3,\langle\cdot,\cdot\rangle)$ by $\e$, that is, the Euclidean space.

We recall that a surface is a two dimensional manifold in $\e$, which we generally denote by $\M$. Given a point $p\in\M$, the tangent and the co-tangent spaces to $\M$ at $p$ are denoted by $T_p\M$ and $T^\ast_p\M$, respectively. 

Let $\{e_1,e_2\}$ be vector (sufficiently differentiable) valued functions on $\M$, such that at each point $p\in\M$, we have: $\{e_1,e_2\}$ is orthonormal with respect to inner product $\langle\cdot,\cdot\rangle$; $\Span\{e_1,e_2\}=T_p\M$; $\{\omega_1,\omega_2\}$ is the dual bases of $\{e_1,e_2\}$. In particular, $\Span\{\omega_1,\omega_2\}=T_p^\ast\M$.

Let $\{\omega_1,\omega_2\}$ be the corresponding dual of the basis $\{e_1,e_2\}$. Since $\langle e_i,e_j\rangle$ is either $0$ or $1$, depending on whether $i=j$ or not, we have
\bb\label{2.2.1}
\langle d e_i,e_j\rangle+\langle e_i,de_j\rangle=0,
\ee
where $d(\cdot)$ denotes the usual differential, and we can then define one-forms
\bb\label{2.2.2}
\omega_{ij}=\langle de_i,e_j\rangle,
\ee
called {\it connection forms}, and from \eqref{2.2.1} we see that $\omega_{ij}=-\omega_{ji}$. 

A one-form $\omega$ can be written as $\omega=f(x,t)dx+g(x,t)dt$, where $f$ and $g$ are certain functions, called coefficients of the form $\omega$. We say that $\omega$ is of class $C^k$ if and only if both $f$ and $g$ are $C^k$ functions.

Let $\otimes$ and $\wedge$ be the tensor and wedge products (for further details, see \cite[page 39]{cle}, respectively. The dual forms $\omega_1$ and $\omega_2$, jointly with the connection forms, satisfy the following relations: 
\bb\label{2.2.3}
d\omega_1=\omega_2\wedge\omega_{21},\quad d\omega_2=\omega_1\wedge\omega_{12},
\ee
\bb\label{2.2.4}
\omega_1\wedge\omega_{13}+\omega_2\wedge\omega_{23}=0,
\ee
and
\bb\label{2.2.5}
d\omega_{12}=\omega_{13}\wedge\omega_{32},\quad d\omega_{13}=\omega_{12}\wedge\omega_{23},\quad d\omega_{23}=\omega_{21}\wedge\omega_{13}.
\ee

It is important to note that the connection form $\omega_{12}$ is completely determined by the forms $\omega_1$ and $\omega_2$, and is known as the Levi-Civita (connection form). For this reason, it is common write $\omega_3:=\omega_{12}$. Moreover, we can define the Gaussian curvature as being the function ${\cal K}$ satisfying the relation
\bb\label{2.2.6}
d\omega_3=-{\cal K}\,\omega_1\wedge\omega_2.
\ee

Equation \eqref{2.2.6} is called the {\it Gauss equation}, and it reflects the fact that the Gaussian curvature is intrinsically determined by the surface, whereas we can rewrite equations \eqref{2.2.3} in terms of form $\omega_3$, which reads
\bb\label{2.2.7}
d\omega_1=\omega_3\wedge\omega_{2},\quad d\omega_2=\omega_1\wedge\omega_{3}.
\ee
Equations \eqref{2.2.6}-\eqref{2.2.7} are called {\it structure equations} of the surface $\M$.

\begin{definition}\label{def2.1.2}
Let $\omega_1$, $\omega_2$, $\omega_{13}$, and $\omega_{23}$ be given one-forms on a surface $\M$ in $\e$, such that $\{\omega_1,\omega_2\}$ is LI, and $p\in\M$. The first and second fundamental forms of $\M$ are defined, on each $T_p\M$, by
$I(v)=\omega_1(v)^2+\omega_2(v)^2$ and $II(v)=\omega_{13}(v)\omega_1(v)+\omega_{23}(v)\omega_2(v)$, for each $v\in T_p\M$.
\end{definition}

Commonly one writes the first and the second fundamental forms as $
I=\omega_1^2+\omega_2^2$ and $II=\omega_{13}\omega_1+\omega_{23}\omega_2$,
with the convection $\al\be=\al\otimes\be$ and $\al^2=\al\al$, for any (one-)forms $\al$ and $\be$.

We observe that everything done so far refers to a given surface $\M$ in the euclidean space $\e$. A quite useful result for our purposes is
\begin{lemma}\label{lemma2.1}
Let $\omega_1$, $\omega_2$, $\omega_{12}$, $\omega_{13}$, and $\omega_{23}$ be $C^1$ one-forms. Then they determine a local surface up to a euclidean motion if and only if $\omega_1\wedge\omega_2\neq0$ and equations \eqref{2.2.3}--\eqref{2.2.5} are satisfied.  
\end{lemma}

Lemma \ref{lemma2.1} (see \cite[Theorem 10-19, page 232]{gug} and also \cite[Theorem 10-18, page 232]{gug} for its proof) is a {\it sine qua non} result from classical differential geometry of surfaces that enabled us to consider solutions $u\in C^{3,1}(\R\times(0,\infty))$ and use them to prove our theorem \ref{thm1.1}.

\begin{remark}\label{rem2.1}
The relevance of lemma \ref{lemma2.1} for us is the following: it states that if a set of given one-forms in $\R^3$ satisfies its conditions, then they define, at least locally, a surface $\M$ in the euclidean space. Such a result is sometimes fundamental theorem of surface theory, see \cite[theorem 11, page 143]{ilka}, or also called Bonnet theorem, see \cite[theorem 4.39, page 127]{cle} or \cite[theorem 4.24, page 153]{ku}. 
\end{remark}

Surfaces for which their Gaussian curvatures are constant and negative are called {\it pseudospherical} surfaces \cite[page 9]{cle}.

\subsection{Equations describing pseudospherical surfaces}

If we take ${\cal K}=-1$ in the structure equations \eqref{2.2.6}--\eqref{2.2.7}, we then have
\bb\label{2.3.1}
d\omega_1=\omega_3\wedge\omega_2,\quad d\omega_2=\omega_1\wedge\omega_3,\quad d\omega_3=\omega_1\wedge\omega_2.
\ee
Sasaki's observation \cite{sasaki} can be summed up as follows: if we denote 
\bb\label{2.3.2}\omega_i=f_{i1}dx+f_{i2}dt,\quad 1\leq i\leq 3,
\ee
from the AKNS method \cite{akns} we can determine functions $f_{ij}$ for which the corresponding triad of one-forms satisfies \eqref{2.2.7} and, as a consequence, they determine a PSS in an intrinsic way.

Let $(x,t)$ be independent variables. A differential equation for a real valued function $u=u(x,t)$ of order $n$ is generically denoted by 
\bb\label{2.3.3}
{\cal E}(x,t,u,u_{(1)},\cdots,u_{({n)}})=0.
\ee
\begin{definition}\label{def2.2}
A differential equation \eqref{2.3.3} is said to describe a pseudospherical surface, or it is said to be of pseudospherical type, if it is a necessary and sufficient condition for the existence of differentiable functions $f_{ij}$, $1\leq i,j\leq 3$, such that the forms \eqref{2.3.2} satisfy the structure equations of a pseudospherical surface \eqref{2.3.1}.
\end{definition}

In practical terms, given a triad of one-forms $\omega_1$, $\omega_2$ and $\omega_3$, and an equation \eqref{2.3.3}, we can check if they describe a PSS surface in the following way: let us define a matrix of one-forms $\Omega$ by
\bb\label{2.3.4}\Omega=\f{1}{2}\begin{pmatrix}
\omega_2 & \omega_1-\omega_3 \\
\omega_1-\omega_3 & -\omega_2
\end{pmatrix}=:(\Omega_{ij}),\quad (\Omega\wedge\Omega)_{ij}:=(\sum_{k=1}^2\Omega_{ik}\wedge\Omega_{kj}),
\ee
\bb\label{2.3.5}
\Sigma:=d\Omega-\Omega\wedge\Omega,\quad d\Omega:=(d\Omega_{ij}).
\ee

If, when restricted to the manifold determined by the solutions of \eqref{2.3.3}, the matrix $\Sigma$ vanishes, we say that \eqref{2.3.3} is a PSS equation, and the triad $\{\omega_1,\, \omega_2,\, \omega_3\}$ satisfies the structure equations of a PSS equation with Gaussian curvature ${\cal K}=-1$.

\begin{example}\label{example2.1} 
Let $m_1\in\{-2,1\}$, $\mu\in\R$ and consider the triad of one-forms
\bb\label{2.3.6}
\ba{lcl}
\omega_1&=&\Big(u-u_{xx}\Big)dx+\Big(2u(u-u_{xx})+\psi\Big)dt,\\
\\
\omega_2&=&\Big(\mu (u-u_{xx})\pm m_{1}\sqrt{1+\mu^{2}}\Big)dx+\ds{\mu\big(2u(u-u_{xx})+\psi\big)}dt,\\
\\
\omega_3&=&\Big(\pm\sqrt{1+\mu^{2}}(u-u_{xx})+m_{1}\mu \Big)dx\\
\\
&&\pm\Big(\ds{\sqrt{1+\mu^{2}}\big(2u(u-u_{xx})+\psi\big)}\Big)dt,
\ea
\ee
where
\bb\label{2.3.7}
\psi:=\frac{4}{m_{1}}uu_x-2u_x^{2}-2u^{2},
\ee
and
\bb\label{2.3.8}
{\cal E}=u_{t}-u_{txx}-4uu_{x}-2u_{x}^{2}-2uu_{xx}+6u_{x}u_{xx}+2uu_{xxx}.
\ee

Note that ${\cal E}=0$ is nothing but \eqref{1.0.1}. It is straightforward, but lengthy, to confirm that (see \cite[Theorem 4.5]{tito})
\bb\label{2.3.9}
\ba{lcl}
d\omega_1-\omega_3\wedge\omega_2&=&{\cal E}dx\wedge dt,\quad
d\omega_2-\omega_1\wedge\omega_3={\cal E}dx\wedge dt,\\
\\
d\omega_3-\omega_1\wedge\omega_2&=&\pm\sqrt{1+\mu^2}{\cal E}dx\wedge dt.
\ea
\ee

Substituting \eqref{2.3.6}--\eqref{2.3.7} into \eqref{2.3.4}, after reckoning we conclude that \eqref{2.3.5} is given by
\bb\label{2.3.10}
\Sigma=\f{{\cal E}}{2}\begin{pmatrix}
1 & 1\pm\sqrt{1+\mu^2} \\
1\pm\sqrt{1+\mu^2} & -1
\end{pmatrix}.
\ee
Therefore, $\Sigma=0$ if and only if ${\cal E}=0$, but ${\cal E}=0$ if and only if the forms \eqref{2.3.6} satisfies the structure equations \eqref{2.3.1} for a PSS with ${\cal K}=-1$ in view of \eqref{2.3.9}. In particular, $u$ must be a solution of \eqref{1.0.1}.
\end{example}

It is important to highlight that a {\it sine qua non} condition for the existence of a PSS defined on an open set $\Omega$ contained in the domain of a solution $u$ is that $\omega_1\wedge\omega_2\neq0$ whenever $(x,t)\in\Omega$, otherwise the Gaussian curvature cannot be inferred from the Gauss equation \eqref{2.2.6}.

\begin{definition}\label{def2.3} Suppose that \eqref{2.3.3} is a PSS equation with corresponding one forms satisfying \eqref{2.3.1}. A solution $u$ of \eqref{2.3.3} for which $\omega_1\wedge\omega_2\neq0$ is called {\it generic}, whereas those satisfying $\omega_1\wedge\omega_2=0$ are said to be non-generic.
\end{definition}

\begin{example}\label{example2.2}
From the one-forms $\omega_1$ and $\omega_2$ in \eqref{2.3.6}, we obtain
\bb\label{2.3.11}
\omega_1\wedge\omega_2=\pm\sqrt{1+\mu^2}\Big(2m_1uu_{xx}-4uu_x+2m_1u_x^2\Big)dx\wedge dt.
\ee

The condition $\omega_1\wedge\omega_2=0$ on an open set $\Omega$ contained in the domain of $u$ is satisfied in the following circumstances:
\begin{itemize}
    \item For $m_1=-2$, then $\phi(x,t)=\pm\sqrt{ae^{-x}+b}$;
    \item For $m_1=1$, then $\phi(x,t)=\pm\sqrt{ae^{2x}+b}$ or $\phi(x,t)=f(t)e^{x}$.
\end{itemize}

Above $a$ and $b$ are real constants, whereas $f\in C^1(\R)$.

Let $u$ be a solution of \eqref{1.0.1}, satisfying the condition $u(x+1,t)=u(x,t)$. Then it is non-generic on an open set $\Omega$ if and only if it is constant. 
\end{example}

\subsection{Sobolev spaces and a few of functional analysis}
    
Let ${\cal P}[0,1]$ be the collection of all periodic functions $f:\R\rightarrow\mathbb{C}$ with period $1$. Given a positive integer $k$, we denote by $f^{(n)}$ its $n-$th order derivative, while the set of functions $f$ for which $f^{(n)}\in C^0(\R)$, $0\leq n\leq k$, is denoted by $C^k(\R)$. If $k$ is a non-negative integer, we define $\C^k[0,1]=C^k(\R)\cap{\cal P}[0,1]$. For the very particular case $k=\infty$, we write ${\cal P}$ instead of $\C^\infty[0,1]$, with topological dual denoted by ${\cal P}'$. Recall that a member of ${\cal P}'$ is a continuous linear functional $f:{\cal P}\rightarrow\mathbb{C}$.

The Fourier transform of $f\in{\cal P}'$ is defined by
$$\hat{f}(k)=\f{1}{2\pi}\int_{0}^1 f(x)e^{-ikx}dx.$$
Let $\ell^2(\mathbb{Z})$ be the collection of sequences $\al=(\al_n)_{n\in\mathbb{Z}}$ such that
$$\sum_{n\in\mathbb{Z}}|\al_n|^2<\infty.$$

Given $s\in\R$, we denote by $\ell^2_s(\mathbb{Z})$ as the collection of $\al\in\ell^2(\mathbb{Z})$ such that
$$\sum_{k=-\infty}^\infty (1+|k|^2)^s|\al_k|^2<\infty,$$
which has a structure of a Banach space when endowed with norm
$$\|\al\|_{\ell^2_s}=\sqrt{\sum_{k=-\infty}^\infty (1+|k|^2)|\al_k|^2}.$$

The periodic Sobolev space of order $s$ is
$H^s_{\text{per}}=\{f\in{\cal P}';\,\,(\hat{f}(k))_{k\in\mathbb{Z}}\in\ell_s^2(\mathbb{Z})\}$,
and the sesquilinear form
$$\big(f\big|g\big)_s=\sum_{k=-\infty}^\infty(1+|k|^2)^2\hat{f}(k)\overline{\hat{g}(k)}$$
turns it into a Hilbert space. In particular, note that $H^0_{\text{per}}=L^2[0,1]$.

Let us now define the following equivalence relation: given $a,b\in\R$, we say that $a\sim b$ if $b=a+k$, for some integer $k$. The quotient space $\R/\sim$ can be identified with the set $[0,1)$, which we shall denote by $\s$. For this reason, henceforth we define $H^s(\s):=H_{\text{per}}[0,1]$. The norm of $H^s(\s)$ will be denoted by $\|\cdot\|_s$, whereas $\|\cdot\|_\infty$ is reserved for the norm in $L^\infty$.

Given two Banach spaces $X$ and $Y$, we write $X\hookrightarrow Y$ to mean that $X$ is continuously and densely embedded in $Y$.

\begin{lemma}{\tt(\cite[Theorem 3.193, page 201]{iorio})}\label{lemma2.2}
Let $s,r\in\R$, with $s\geq r$. Then $H^s(\s)\hookrightarrow H^r(\s)$ and
$\|f\|_r\leq\|f\|_s,$
for any $f\in H^s(\s)$. In particular, $H^s(\s)\hookrightarrow L^2[0,1]$ for any $s\geq0$.
\end{lemma}

The next result is known as Sobolev Lemma, or also as Sobolev Embedding Theorem.

\begin{lemma}{\tt(\cite[Theorem 3.195, page 204]{iorio}, \cite[Proposition 3.3, page 329]{taylor})}\label{lemma2.3} 
If $s>1/2$, then $H^s(\s)\hookrightarrow C_{\text{per}}^0[0,1]$ and $\|f\|_\infty\leq c\|f\|_s$, for some constant $c$  depending only on $s$, where $f\in H^s(\s)$. More generally, if $s>1/2+m$, where $m$ is a positive integer, then $H^s(\s)\hookrightarrow \C^m[0,1]$. 
\end{lemma}

We conclude our revision on Sobolev spaces by recalling the algebra property.
\begin{lemma}{\tt(\cite[Theorem 3.200, page 207]{iorio})}\label{lemma2.4}
If $s>1/2$, for any $f,g\in H^s(\s)$, we have $fg\in H^s(\s)$ and their norm satisfies the estimate
$\|fg\|_s\leq c\|f\|_s\|g\|_s,$
for some constant $c>0$ depending only on $s$.
\end{lemma}

\subsection{Semigroup Approach} \label{sec2.4}

Let us revisit basic aspects of Kato's theory \cite{KatoI,KatoII,Pazy}, also known as semigroup approach, which is our main tool for proving well-posedness of solutions with the regularity we need to make them consistent with the geometric nature of our problem.

Let $X$ be a Hilbert space, and let $u(\cdot,t)\in X$ such that 
  \begin{equation}
  \label{qlee}
     u_t +A(u)u = f(u), ~~~~t\geq 0,~~~~~u(0)=u_0.
  \end{equation}
Let $Y\hookrightarrow X$ and $S:Y\rightarrow X$ be a topological isomorphism. Assume that

\begin{itemize}
\item[(A1)]  For any given $r>0$ it holds that for all $u \in  \mathrm B_r(0) \subseteq Y$ (the ball around the origin in $Y$ with radius $r$), the linear operator $A(u)\colon X \to X$ generates a strongly continuous semigroup $T_u(t)$ in $X$ which satisfies $\| T_u(t) \|_{\mathcal L(X)} \leq \mathrm e^{\omega_r t}$, for all $t\in [0,\infty)$, for a uniform constant $\omega_r > 0$;
\item[(A2)] $A$ maps $Y$ into $\mathcal L(Y,X)$, more precisely the domain $D(A(u))$ contains $Y$ and the restriction $A(u)|_Y$ belongs to $\mathcal L(Y,X)$ for any $u\in Y$. Furthermore $A$ is Lipschitz continuous in the sense that for all $r>0$ there exists a constant $C_1$ which only depends on $r$ such that
$
\| A(u) - A(v) \|_{\mathcal L(Y,X)} \leq C_1 \, \|u-v\|_X 
$
for all $u,~v\in\mathrm B_r(0) \subseteq Y$.
\item[(A3)] For any $u\in Y$ there exists a bounded linear operator $B(u) \in \mathcal L(X)$ satisfying $B(u) = S A(u) S^{-1} - A(u)$ and $B \colon Y \to \mathcal L(X)$ is uniformly bounded on bounded sets in $Y$. Furthermore for all $r>0$ there exists a constant $C_2$ which depends only on $r$ such that  
$
\| B(u) - B(v)\|_{\mathcal L(X)} \leq C_2 \, \|u-v\|_Y,
$
for all $u,~v \in \mathrm B_r(0)\subseteq Y$;
\item[(A4)] 
For all $t\in[0,\infty)$, $f$ is uniformly bounded on bounded sets in $Y$. Moreover, the map $f\colon Y \to Y$ is locally $X$-Lipschitz continuous in the sense that for every $r>0$ there exists a constant $C_3>0$, depending only on $r$, such that
$\| f(u) - f(v)\|_{X} \leq C_3 \, \|u-v\|_X$, for all $u,~v \in \mathrm B_r(0) \subseteq Y$,
and locally $Y$-Lipschitz continuous in the sense that for every $r>0$ there exists a constant $C_4>0$, depending only on $r$, such that
$\| f(u) - f(v)\|_{Y} \leq C_4$ $\|u-v\|_Y$, for all $u,~v \in \mathrm B_r(0) \subseteq Y$.
\end{itemize}

\begin{lemma} \cite{KatoI}
\label{kato}
Assume that (A1)-(A4) hold. Then for given $u_0\in Y$, there is a maximal time of existence $T>0$, depending on $u_0$, and a unique solution $u$ to (\ref{qlee}) in $X$ such that
$
u=u(u_0,.)\in C^0(Y,[0,T))\cap C^1(X,[0,T)).
$
Moreover, the solution depends continuously on the initial data,
i.e. the map $u_0\rightarrow u(u_0,.) $ is continuous from $Y$ to \mbox{$C^0(Y,[0,T))\cap C^1(X,[0,T))$}.
\end{lemma}

\section{Proof of theorem \ref{thm1.2}}\label{sec3}

We begin by noticing that the equation in (\ref{1.0.5}) is in the quasi-linear equation form (\ref{qlee}), where
\bb \label{ql}
 A(u) = -2u\partial_x
\ee
and 
\bb \label{nl}
 f(u) =  \Lambda^{-2} \partial_x \big( u^2 + (u^2)_x\big) .
\ee

Let 
$$(\Lambda^sf)(k):=\sum_{k\in\mathbb{Z}}(1+n^2)^{s/2}\hat{f}(n)e^{ink}.$$

For any $s,s'\in\R$, $\Lambda^s:H^{s'}(\s)\rightarrow H^{s'-s}(\s)$ is an isomorphism \cite[page 330]{taylor}. Then, it is natural to choose as Hilbert spaces $X\coloneqq (H^{s-1}(\mathbb{S}), \|\cdot \|_{s-1})$ and $Y\coloneqq (H^{s}(\mathbb{S}),\|\cdot \|_s)$ with $s>\frac{3}{2}$, and $S=\Lambda$ as well, and work with them using Kato's approach. We aim at proving lemmas ensuring the validity of the assumptions (A1)-(A4). For convenience, in the remaining part of this section we simply write $H^s$ in place of $H^s(\s)$. Moreover, for a given function $g\in H^r$ with $r>1/2$ let us denote by $M_g$ the corresponding multiplication operator on $H^r$, i.e.~$M_g \colon H^r \to H^r,  w\mapsto g w$. Since $H^r$, $r>1/2$, is closed under multiplication, $M_g$ is continuous.

Now, we verify the assumptions needed for Theorem \ref{thm1.2}. We start with assumption (A1):
  \begin{lemma} \label{opA}
Let $s>3/2$. For any given $r>0$, it holds that for all \mbox{$u \in  \mathrm B_r(0) \subseteq H^{s}$}, the linear operator $A(u)\colon H^{s-1} \to H^{s-1}$, with domain $D(A(u))\coloneqq \{ w \in H^{s-1}\colon A(u) w \in H^{s-1} \}$,
generates a strongly continuous semigroup $T_u(t)$ in $X$ which satisfies
$\| T_u(t) \|_{\mathcal L(X)} \leq \mathrm e^{\omega_r t}$ for all $t\in [0,\infty)$, for a uniform constant $\omega_r > 0$. In particular, the operator $A(u)$ given in \eqref{ql}, with domain
$\mathcal{D}(A)=\{\omega\in H^{s-1}:A(u)\omega\in H^{s-1}\}\subset H^{s-1}$
is quasi-m-accreative in $H^{s-1} $if $u\in H^{s}$, $s>\frac{3}{2}$.
\end{lemma}

For convenience, it would be good to mention that the coefficient in \eqref{ql} does not affect the analysis. It just plays a role in constant estimation which is not of our interest. Therefore, we will neglect it and keep the operator form as $u\partial_x$. We prove this lemma in two steps.  First step is given as follows:

\begin{lemma}\label{pre}
The operator $A(u)= u\partial_x$ in $L^2$, with $u\in H^s$, $s>\frac{3}{2}$, is\\ \mbox{quasi-m-accreative}.
\end{lemma}

\begin{proof}
A linear operator $A=A(u)$ in $X$ is \mbox{quasi-m-accretive} if and only if \cite{KatoII}:
 \begin{enumerate}
  \item [(a)] There is a real number $\beta$ such that $(A\omega,\omega)_X\geq -\beta\|\omega\|_X^2$ for all $\omega\in D(A)$;
  \item [(b)] The range of $A(u)+\lambda I$ is all of $X$ for some (or equivalently, all) $\lambda>\beta$.
 \end{enumerate}

Note that if the above property (a) holds, then $A+\lambda I$ is dissipative for all $\lambda>\beta$. Moreover, if $A$ is a closed operator, then $A+\lambda I$ has closed range in $X$ for all $\lambda>\beta$. Hence, in order to prove (b) in such a case, it is enough to show that $A+\lambda I$ has dense range in $X$ for all $\lambda>\beta$.  

First we show that $A$ is a closed operator in $L^2$.
Let $(v_n)_{n\in\N}$ be a sequence in $D(A)$ with $v_n\to v$ in $L^2$ and $A v_n \to w$ in $L^2$. Then $u v_n \in H^{1}$ for all $n\in \N$ by definition of $D(A)$ since an alternative way of writing the domain is \mbox{$D(A)=\{\omega\in L^{2}:u\omega\in H^{1}\}$} and $v_n\in D(A)$. Moreover, both $u v_n \to u v$ and $u_x v_n \to u_x v$ in $L^2$ by the continuity of the multiplication $H^r \times L^2 \to L^2$ for $r>1/2$. Therefore,$(u v_n)_x \to w+u_x v$ in $L^2$. Having sequences $(uv_n)_{n\in\mathbb{N}}$ and $((u v_n)_x)_{n\in\mathbb{N}}$ convergent in $L^2$ implies that $(uv_n)_{n\in\mathbb{N}}$ converges in $H^1$
with the limit $ uv$, thus $v\in D(A)$. Moreover the continuity of $\partial_x\colon H^1 \to L^2$ implies that $\lim_{n\to\infty} (u v_n)_x = (u v)_x $, therefore $w=(uv)_x-u_x v = A v$.

Now, we take the following $L^2$ inner product
\begin{align}
 (A(u)\omega,\omega)_{0} &=(u\partial_x\omega ,\omega)_{0} \nonumber
\end{align}
We refer to Lemma \ref{L-core} to be stated below and use integration by parts to get:
\begin{align*}
|(u\partial_x\omega ,\omega)_{0}|
=|-\frac{1}{2}(u_x, \omega^2)_0|\leq C \|u_x\|_{L^\infty}\|w\|_{0}^2\leq \tilde{C}\|\omega\|_{0}^2.
\end{align*}

Having $\|u\|_s$ bounded allows us to choose $\beta=\tilde{C}(\|u\|_{H^{s}})$ and to show that the operator satisfies the inequality in (a). Thus, $A(u)+\lambda I$ is dissipative for all $\lambda>\beta$. Moreover, recall that $A(u)$ is a closed operator. Therefore, we now show that $A(u)+\lambda I$ has dense range in $L^2$ for all $\lambda>\beta$. 

It is known that if the adjoint of an operator has trivial kernel, then the operator has dense range \cite{RS}. For $A(u)=u\partial_x $, the adjoint operator can be expressed $A^*(u)=-u_x-u\partial_x$.

Observe that 
\begin{displaymath}
A^*(u)\omega=-u_x\omega-u\omega_x=-(u\omega)_x.
\end{displaymath}
Since $u_x\in L^{\infty}$ and $\omega\in L^2$, we have $u_x\omega\in L^2$. Having also $A(u)\omega=u\omega_x\in L^2$ for $\omega\in D(A)$ reveals that
$\mathcal{D}(A^*)=\{\omega\in L^2:A^{*}(u)\omega\in L^2\}$.

Assume that $A(u)+\lambda I$ does not have a dense range in $L^2$. Then, there exists $0\neq z\in L^2$
such that $((A(u)+\lambda I)\omega,z)_0=0$ for all $\omega\in \mathcal{D}(A)$. Since $H^1\subset \mathcal{D}(A)$, $\mathcal{D}(A)=\mathcal{D}(A^*)$ is dense in $L^2$. It means that there exists a sequence $z_k\in \mathcal{D}(A^*)$ such that it converges to an element $z\in L^2$. Recall that $D(A^*)$ is closed. So, $z\in \mathcal{D}(A^*)$. Moreover,
\begin{equation*}
((A(u)+\lambda I)\omega,z)_{0}=(\omega,(A(u)+\lambda I)^*z)_{0}=0
\end{equation*}
reveals that $(A^*(u)+\lambda I) z=0$ in $L^2$. Multiplying by $z$ and integrating by parts, we get
\begin{equation*}
0=((A^*(u)+\lambda I)z,z)_{0}=(\lambda z,z)_{0}+(z,A(u)z)_{0}\geq (\lambda-\beta)\|z\|_{0}^2~~~\forall \lambda>\beta
\end{equation*}
and thus, $z=0$, which contradicts our assumption. It completes the proof of (b). Therefore, the operator $A(u)$ is quasi-m-accreative.
\end{proof}

\noindent
In the proof of Lemma \ref{opA} we use the fact that $C^{\infty}(\mathbb{S})$ is a \emph{core} for $A$ in $H^{s-1}$, i.e. $A(u)v$ can be approximated by smooth functions in $H^{s-1}$  (\cite{CE2}): 
\begin{lemma}
\label{L-core}
    Given $v\in D(A)$ there exists a sequence $(v_n)_{n\in \mathbb{N}}$ in $\mathcal C^{\infty}$ such that both $v_n \to v$ and $A v_n \to Av$ in $H^{s-1}$.
\end{lemma}

\begin{proof}
Let $v\in D(A)$ and fix $\rho\in C_c^\infty$, where $C_c^\infty$ denotes the set of $C^\infty$ functions with compact support, with $\rho \geq 0$ and $\int_{\R} \rho =1$. Given $n\geq 1$, let $\rho_n= n\rho(nx)$. If we set $v_n\coloneqq \rho_n * v$, then $v_n \in C_c^\infty$ for $n\geq 1$ and $v_n \to v$ in $H_p^{s-1}$.
We have to prove that $(uv_n)_x \to (uv)_x$ in $H^{s-1}$. Since $v\in D(A)$, we have that $uv_x\in H^{s-1}$ and hence 
 $\rho_n *(uv_x)\rightarrow uv_x$. Moreover, since $uv_n\in H^s$ it follows that $(uv_n)_x = u_x v_n + u(v_n)_x \in H^{s-1}$,  hence we have that  $u_xv_n\rightarrow u_x v$ in $H^{s-1}$. Therefore
\begin{align*}
    (uv_n)_x-(uv)_x&=u_xv_n-u_xv +\rho_n *(uv_x)-uv_x + u(v_n)_x - \rho_n*(uv_x)
\end{align*}
holds true and it suffices to show that $ u(v_n)_x-\rho_n*(uv_x) \to 0$ in $H^{s-1}$. To this end, denote 
$$P_n v \coloneqq u(v_n)_x-\rho_n*(uv_x), \quad  n\geq 1.$$ 
We will show   that there exists $K>0$ independent of $v$ such that 
\bb
\label{eq-Pn} 
\|P_n v \|_{s-1} \leq K\|v\|_{s-1}, \quad n\geq 1. 
\ee
That will enable us to conclude that $P_n$ is uniformly bounded in $H^{s-1}$ by the uniform boundedness principle. When we approximate $v$ in $H^{s-1}$ by smooth functions, and use this conclusion, we will be able to prove the assertion $P_n\to 0$ for $v\in C_c^\infty$. Since the set of smooth functions is dense in $H^{s-1}$ and $P_n$ are uniformly bounded, the proof will be completed. 

We first notice that 
\begin{align*}
P_n v(x) &= \int_{\R}(\rho_n)_y(y) (u(x) - u(x-y))v(x-y) dy +(\rho_n*(u_x v))(x)\\
             &= n^2\int_{\R} \rho_y(ny) (u(x)-u(x-y)) v(x-y)dy+(\rho_n*(u_x v))(x)\\
             &= n\int_{-1}^1 \rho_y(y) (u(x)-u(x-\frac{y}{n})) v(x-\frac{y}{n})dy+(\rho_n*(u_x v))(x),
\end{align*}
where $\rm supp(\rho)\subset [-1,1]$. Moreover, using the mean value theorem, we obtain the estimate 
\begin{align*}
&\Big|n^2\int_{\R} \rho_y(ny) (u(x)-u(x-y)) v(x-y)dy\Big|= \Big|n^2\int_{\R} \rho_y(ny) u_x(x_0) y v(x-y)dy\Big|\\
        &= \Big|\int_{-1}^1\rho_y(y) u_x(x_0) y v(x-y)dy\Big|\leq \|u_x\|_{L^\infty}\int_{-1}^1 |\rho_y(y)|\,|y|\,|v(x-\frac{y}{n})|dy,
\end{align*}
for some $x_0\in (x,x-y)$. Let now  $C\coloneqq \sup_{x\in\mathbb{R}}\|u_x\|_{L^\infty}^2\int_{-1}^1 |\rho_y(y)y|^2dy$. Then the Cauchy-Schwarz inequality, Fubini's theorem and the fact that the operator $\Lambda^{s-1}$ commutes with integration yield that
\begin{align*}
&\|n^2\int_{\R} \rho_y(ny) (u(x)-u(x-y)) v(x-y)dy\|^2_{s-1}\\
&=\|\Lambda^{s-1}\int_{-1}^1 \rho_y(y) u_x(x_0)y (v(x-\frac{y}{n}))dy\|^2_{2}\\
&=\int_{\R} \left|\int_{-1}^1 \rho_y(y) u_x(x_0)y \Lambda^{s-1}v(x-\frac{y}{n})dy\right|^2 dx\\
&\leq C \int_{-1}^1\int |\Lambda^{s-1}v(x-\frac{y}{n})|^2 dx dy\leq 2C\|v\|_{s-1}.
\end{align*}
Moreover, we obtain by Plancherel's theorem that 
\begin{align*}
\|\rho_n*(u_x v)\|_{s-1} &= \|\Lambda^{s-1}(\rho_n*(u_x v))\|_{2} 
        = \|\rho_n*\Lambda^{s-1}(u_x v))\|_{2}
    &\leq     \|\Lambda^{s-1}(u_x v)\|_2\\
    &\leq  \|u_x\|_{L^\infty}\|v\|_{s-1}.
\end{align*}
Therefore we conclude that
\begin{equation}\label{bound}
\|P_n v\|_{s-1}\leq (\sqrt{2C}+\|u_x\|_{L^\infty})\,\|v\|_{s-1},~~n\geq 1.
\end{equation}
For $K=\sqrt{2C}+\|u_x\|_{L^\infty}$ in (\ref{eq-Pn}), proof is completed by the estimate (\ref{bound}).
\end{proof}

Second step to prove Lemma \ref{opA} is making use of the following lemma proved in \cite{Pazy}:

\begin{lemma}\label{adm}
Let $X$ and $Y$ be two Banach spaces such that $Y$ is continuously and densely embedded in $X$. Let $-A$ be the infinitesimal generator of the $C_0$-semigroup $T(t)$ on $X$ and let $Q$ be an isomorphism from $Y$ onto $X$. Then $Y$ is $-A$-admissible (i.e. $T(t)Y\subset Y$ for all $t\geq 0$, and the restriction of $T(t)$ to $Y$ is a $C_0$-semigroup on $Y$) if and only if $-A_1=-QAQ^{-1}$ is the infinitesimal generator of the $C_0$-semigroup $T_1(t)=QT(t)Q^{-1}$ on $X$. Moreover, if $Y$ is $-A$-admissible, then the part of $-A$ in $Y$ is the infinitesimal generator of the restriction $T(t)$ to $Y$.  
\end{lemma}

Before we proceed with the proof of Lemma \ref{opA}, we give the commutator estimate which will be used:
\begin{lemma}[\cite{Taylor}]
\label{L-comm} 
Let $m> 0$, $s\geq 0$ and $3/2<s+m\leq \sigma$. Then for all $f\in H^{\sigma}$ and $g\in H^{s+m-1}$ one has
$
\|[\Lambda^{m},f]g\|_{s} \leq C\|f\|_{\sigma} \, \|g\|_{s+m-1},
$
where $C$ is a constant which is independent of $f$ and $g$.
\end{lemma}

\noindent
\textbf{Proof of Lemma \ref{opA}}:
Following the arguments in the proof of Lemma \ref{pre}, we first take the following $H^{s-1}$ inner product
\begin{align}
 (A(u)\omega,\omega)_{s-1} &=(u\partial_x\omega ,\omega)_{s-1}=(\Lambda^{s-1}u\partial_x\omega ,\Lambda^{s-1}\omega)_{0} \nonumber \\
&=([\Lambda^{s-1},u]\partial_x\omega ,\Lambda^{s-1}\omega)_{0}+(u\partial_x\Lambda^{s-1}\omega ,\Lambda^{s-1}\omega)_{0} \label{est}.
\end{align}
Using Cauchy-Schwartz's inequality and Lemma \ref{L-comm} with $m=s-1$, $\sigma=s$, we get the following estimate for the first term of (\ref{est}):
\begin{align*}
|([\Lambda^{s-1},u]\partial_x\omega ,\Lambda^{s-1}\omega)_{0}|&\leq C\|u\|_s\|\partial_x \omega\|_{s-2}\|\omega\|_{s-1}\leq \tilde{C}\|\omega\|_{s-1}^2,
\end{align*}
for some constant $\tilde{C}$ depending on $\|u\|_s$.

For the second term of (\ref{est}), we again refer to Lemma \ref{L-core} and use integration by parts to get:
\begin{align*}
|(u\partial_x\Lambda^{s-1}\omega ,\Lambda^{s-1}\omega)_{0}|
=|-\frac{1}{2}(u_x, (\Lambda^{s-1}\omega)^2)_0|\leq C \|u_x\|_{L^\infty}\|w\|_{s-1}^2\leq \tilde{C}\|\omega\|_{s-1}^2.
\end{align*}
Choosing $\beta=\tilde{C}(\|u\|_{H^{s}})$, the operator satisfies the required inequality.

Moreover, let $Q:=\Lambda^{s-1}$ and notice that $Q$ is an isomorphism of $H^{s-1}$ to $L^2$ and $H^{s-1}$ is continuously and densely imbedded into $L^2$ as $s>\frac{3}{2}$. \\
Define 
\begin{align*}
    A_1(u)=QA(u)Q^{-1}=\Lambda^{s-1}A(u)\Lambda^{1-s}=\Lambda^{s-1}u\partial_x\Lambda^{1-s}=\Lambda^{s-1}u\Lambda^{1-s}\partial_x,
\end{align*}
and let $\omega\in L^2$ and $u\in H^s$, $s>\frac{5}{2}$. Then write $B_1(u)=A_1(u)-A(u)$ and consider the following estimate:
\begin{align*}
    \|B_1(u)\omega\|_0&=\|[\Lambda^{s-1}, A(u)]\Lambda^{1-s}\omega\|_0=\|[\Lambda^{s-1}, u]\Lambda^{1-s}\partial_x\omega\|_0\\
    &\leq C \|u\|_s\|\Lambda^{1-s}\partial_x\omega\|_{s-2}\leq C\|u\|_s\|\omega\|_0,
\end{align*}
where we applied Lemma \ref{L-comm} with $m=s-1$ and $\sigma=s$. Hence, we obtain $B_1(u)\in \mathcal{L}(L^2).$

Recall from Lemma \ref{pre} that $A(u)$ is quasi-m-accretive in $L^2$, i.e. $-A(u)$ is the infinitesimal generator of a $C_0$-semigroup on $L^2$. Thus, $A_1(u)=A(u)+B_1(u)$ is also the infinitesimal generator of a $C_0$-semigroup in $L^2$ by means of a perturbation theorem for semigroups (see \cite{Pazy}). Lemma \ref{adm} reveals that for $Y=H^{s-1}$, $X=L^2$ and $Q=\Lambda^{s-1}$, $H^{s-1}$ is $A$-admissible. Hence, $-A(u)$ is the infinitesimal generator of a $C_0$-semigroup on $H^{s-1}$. 
\qed

We continue with the proof of assumption (A2):

\begin{lemma}
$A$ maps $H^s$ into $\mathcal L(H^s,H^{s-1})$, more precisely the domain $D(A(u))$ contains $H^s$ and the restriction $A(u)|_{H^s}$ belongs to $\mathcal L(H^s,H^{s-1})$ for any $u\in H^{s}$. Furthermore $A$ is Lipschitz continuous in the sense that for all $r>0$ there exists a constant $C_1$ which only depends on $r$ such that
\begin{equation}\label{A_Lipschitz}
\| A(u) - A(v) \|_{\mathcal L(H^{s},H^{s-1})} \leq C_1 \, \|u-v\|_{s-1} 
\end{equation}
for all $u,~v$ inside $\mathrm B_r(0) \subseteq H^{s}$.
\end{lemma}

\begin{proof}
The operator  $A(u)|_{H^s}$ belongs to $\mathcal L(H^s,H^{s-1})$ for any $u\in H^s$, since $\partial_x \in \mathcal L(H^s,H^{s-1})$ and $M_{u}\in \mathcal L(H^{s-1})$.
To see that the required estimate is satisfied, let $u,v, w \in H^s$ be arbitrary. Then,

\begin{align*}
\|(A(u)-A(v))w\|_{s-1} &=\|(u - v ) \partial_x w\|_{s-1} \leq C\|u-v\|_{s-1} \, \|\partial_x w\|_{s-1} \\
&\leq C  \|u-v\|_{s-1} \, \| w\|_s , 
\end{align*}
where $C$ denotes a generic constant. This shows that for arbitrary $r>0$ one can always find a constant $C_1$ such that \eqref{A_Lipschitz} holds uniformly for all $u,v \in \mathrm B_r(0) \subseteq H^s$.
\end{proof}

The last two assumptions (A3)-(A4) are proved by the help of the following commutator and product estimates stated in \cite[Proposition B.10.(2)]{Lannes} and \cite[Lemma A1]{KatoI}, respectively:
\begin{lemma}\label{L_comm} 
Let $r>1/2$.	
	\begin{enumerate} 
		\item[(i)]  \label{com_est_2}
		If $-1/2 < t \leq r +1$, there exists a constant $C_{r,t}>0$ such that 
		\begin{displaymath}
		    		\big\|[\Lambda^t,M_g]h\big\|_0 \leq C_{r,t} \, \|g\|_{r+1} \, \|h\|_{t-1}
		\end{displaymath}

		for all $g\in H^{r +1}$ and $h\in H^{t-1}$.
		\item  [(ii)] \label{com_est_3}
		If $-r < t\leq r$, there exists a constant $C_{r,t}>0$ such that 
		$\|fg\|_t\leq C_{r,t} \, \|f\|_r \, \|g\|_t$,
		for all $f\in H^{r}$ and $g\in H^{t}$.
	\end{enumerate}
\end{lemma}
Even though Lemma \ref{L_comm} is stated on the real line (in general $\R^m$), it holds on periodic domain as well.

Now, we define a bounded linear operator and prove assumption (A3):
\begin{lemma}
For any $u\in H^{s}$ there exists a bounded linear operator $B(u) \in \mathcal L(H^{s-1})$ satisfying $B(u) = \Lambda A(u) \Lambda^{-1} - A(u)$ and $B \colon H^{s} \to \mathcal L(H^{s-1})$ is uniformly bounded on bounded sets in $H^{s}$. Furthermore for all $r>0$ there exists a constant $C_2$ which depends only on $r$ such that  
\begin{equation}\label{lem_B_estimate}
\| B(u) - B(v)\|_{\mathcal L(H^{s-1})} \leq C_2 \, \|u-v\|_{s}
\end{equation}
for all $u,v \in \mathrm B_r(0)\subseteq H^{s-1}$. Here, $A(u)$ is the operator given by \eqref{ql}.
\end{lemma}

\begin{proof}
Let $u\in H^s$. Since $\partial_x$ commutes with $\Lambda$ and $\Lambda^{-1}$ we obtain that
\begin{equation*}
  B(u)=\Lambda u\partial_x\Lambda^{-1}-u\partial_x
=[\Lambda,u\partial_x]\Lambda^{-1}
=[\Lambda,u]\Lambda^{-1}\partial_x.  
\end{equation*}
Hence we can write $\Lambda^{s-1} B(u)$ as
\begin{align*}
\Lambda^{s-1}[\Lambda,u]\Lambda^{-1}\partial_x&=\Lambda^{s}u\Lambda^{-1}\partial_x-\Lambda^{s-1}u\partial_x=[\Lambda^{s},u]\Lambda^{-1}\partial_x+u\Lambda^{s-1}\partial_x-\Lambda^{s-1}u\partial_x\\
&=[\Lambda^{s},u]\Lambda^{-1}\partial_x +[u,\Lambda^{s-1}]\partial_x.
\end{align*}
Let now $\omega\in H^{s-1}$ and $u,v \in H^s$ be arbitrary. In view of the above identity and Lemma \ref{L_comm}, we obtain the following estimate
\begin{align*}
&\|(B(u)-B(v))\omega\|_{s-1} = \|\Lambda^{s-1}(B(u)-B(v))\omega\|_0 \\
&\leq\|[\Lambda^{s},u-v]\Lambda^{-1}\partial_x\omega\|_0
+\|[u-v,\Lambda^{s-1}]\partial_x\omega\|_0\\
&\leq C \|u-v\|_s(\|\Lambda^{-1}\partial_x\omega\|_{s-1}+\|\partial_x\omega\|_{s-2})\leq C \|u-v\|_s\|\omega\|_{s-1},
\end{align*}
where $C$ is a generic constant independent of $u,w$ and $w$. In particular, this shows that $B(u)$ extends to a bounded linear operator on $H^{s-1}$ for every $u\in H^s$ such that $B\colon H^s \to \mathcal L(H^{s-1})$ is uniformly bounded on bounded sets in $H^s$. Furthermore, this estimation proves that there exists a constant $C_2$ depending only on the radius of the ball $\mathrm B_r(0) \subseteq H^s$ such that \eqref{lem_B_estimate} is satisfied for all $u,v\in B_r(0)$. 
\end{proof}

The last assumption (A4) is proved in Lemma \ref{A4}:
\begin{lemma}\label{A4} 
For all $t\in[0,\infty)$, $f$ is uniformly bounded on bounded sets in $H^s$. Moreover, the map $f\colon H^{s}\to H^{s}$ is locally $H^{s-1}$-Lipschitz continuous in the sense that for every $r>0$ there exists a constant $C_3>0$, depending only on $r$, such that $\| f(u) - f(v)\|_{s-1} \leq C_3 \, \|u - v\|_{s-1}$ for all $u,v \in \mathrm B_r(0) \subseteq H^{s}$
and locally $H^{s}$-Lipschitz continuous in the sense that for every $r>0$ there exists a constant $C_4>0$, depending only on $r$, such that
$\| f(u) - f(v)\|_{s} \leq C_4 \, \|u - v\|_{s}$ for all $u,v \in \mathrm B_r(0) \subseteq H^{s}$.
\end{lemma}

\begin{proof}
Recall that $ f(u) =  \Lambda^{-2} \partial_x \big( u^2 + (u^2)_x\big).$
Therefore, with the help of Lemma \ref{L_comm}
\begin{align*}
\|f(u)-f(v)\|_{s-1} &\leq C \|(u^2-v^2)+(u^2-v^2)_x\|_{s-2}\\
                    &\leq C \|(u+v)(u-v)+((u+v)(u-v))_x\|_{s-2}\\
            		&\leq C (\|(u+v)\|_{s-2}\|(u-v)\|_{s-1}+\|u+v\|_{s-1}\|u-v\|_{s-1})\\
            		&\leq C_3 \|u-v\|_{s-1}
\end{align*}
where $C_3$ is a constant depending on $\|u\|_{H^{s}}$ and $\|v\|_{H^{s}}$. This proves $H^{s-1}$-Lipschitz continuity.

\noindent
Similar arguments will show that we have the following estimates:
\begin{align}\label{lips}
\|f(u)-f(v)\|_{s}\leq& C_4 \|u-v\|_{s}
\end{align}
where $C_4$ is also a constant depending on $\|u\|_{H^{s}}$ and $\|v\|_{H^{s}}$. Since we choose $u_0\in H^{s}$, this estimate actually corresponds to the proof of continuous dependence on the initial data. Note that boundedness of $f(u)$ on bounded subsets \mbox{$\{u\in H^s: \|u\|_s\leq M\}$} of $H^s$ (for all $M$) can be obtained from (\ref{lips}) by choosing $v=0$. Hence, we get the estimates for (A4).
\end{proof}

\section{Proof of theorem \ref{thm1.3}} \label{sec4}

In order to prove theorem \ref{thm1.3} we need some estimates for $u,\,u_x$ and $u_{xx}$. To this end, we need the next result. 

\begin{theorem}\label{thm4.1}
Assume that $u_0\in H^3(\s)$, $m_0(x):=u_0(x)-u''_{0}(x)$, and let $u$ be the corresponding solution of \eqref{1.0.5}. If $m_0(x)\geq0$, $x\in\s$, then $m(x,t):=u(x,t)-u_{xx}(x,t)$ is non-negative for any $t$ as long as the solution exists, and any $x\in\s$. Moreover, $u$ is also non-negative. In particular, if $m_0>0$, then $u>0$.
\end{theorem}
An analogous result for non-periodic problems was proved in \cite[Lemma 5.5]{li-na}, and following the same steps we get the demonstration for Theorem \ref{thm4.1}. For this reason it is omitted.

There is one more fact regarding the $L^\infty$ norm of $u_x$. We begin by observing that 
$$\int_{\s}u_{xx}dx=0.$$
This fact is enough to guarantee the existence of a point $\xi_t-1\in(0,1)$ such that $u_x(t,\xi_t-1)=0$, for each $t\in(0,T)$,.

\begin{lemma}\label{lem4.1}
If $u_0\in H^3(\s)\cap L^1(\s)$, is such that $m_0\geq 0$, then there exists a constant $K>0$ such that the solution of \eqref{1.0.5} satisfies $\|u_x\|_{L^\infty(\s)}\leq K$.
\end{lemma}
\begin{proof}
Let us first assume that $\|m(\cdot,t)\|_{L^1}(\s)$ is constant for any $t$ as long as the solution exists.
    Assume $m_0$ does not change sign and $m_0\geq 0$. Then,
\begin{eqnarray*}
K_1&=&\|m_0\|_{L^1(\s)}=\int_\s m_0(r)dr=\int_\s m(r,t)dr=\int_{\xi_{t}-1}^{\xi_t}m(r,t)dr \\
&\geq&\int_{\xi_{t}-1}^x (u-u_{xx})(r,t)dr=\int_{\xi_{t}-1}^x u(r,t)dr-u_x(x,t)\geq -u_x(x,t)
\end{eqnarray*}
holds for every $x\in[\xi_{t}-1, \xi_t]$. 
Here, we use Theorem \ref{thm4.1}, which guarantee that $u$ does not change sign, under the assumption that $m_0$ does not change sign. Taking into account the final result, we observe that $u_x$ is bounded from below.

Moreover, 
\begin{eqnarray*}
    K_1=\int_{\xi_{t}-1}^{\xi_t}m(r,t)dr
            \geq\int_x^{\xi_{t}} m(r,t)dr=\int_x^{\xi_{t}} u(r,t)dr+u_x(x,t)\geq u_x(x,t).
\end{eqnarray*}
Hence, $u_x$ is bounded also from above. Therefore, we can conclude that $\|u_x\|_{\infty}$ norm is bounded provided that $m$ does not change sign , i.e. $\|u_x\|_{\infty}\leq K$. The case $m_0\leq0$ is proved in a similar way and, therefore, is omitted.

We now complete the demonstration proving that $\|m(\cdot,t)\|_{L^1(\s)}$ is constant. We begin by noticing that \eqref{1.0.1} is itself a conservation law, in the sense that
$$
\p_t(u-u_{xx})=\p_x\Big((2-\p_x)(1+\p_x)u^2\Big)=\p_x\Big((1-\p_x^2)u^2+u^2)\Big).
$$
Integrating the relation above with respect to $x$ on $\s$, we obtain
$$
\f{d}{dt}\int_\s (u-u_{xx})dx=\Big((1-\p_x^2)u^2+u^2)\Big)\big|_\s=0,
$$
meaning that the $\|m(\cdot,t)\|_{L^1(\s)}=const.$. Since $m_0\in L^1(\s)$, we conclude that $\|m(\cdot,t)\|_{L^1(\s)}=\|m_0\|_{L^1(\s)}.$
\end{proof}

Now we start proving Theorem \ref{thm1.3}:

First, we rewrite the equation (\ref{1.0.5}) in the following form
$$u_t-2uu_x+u^2=\Lambda^{-2}(u^2+(u^2)_x)$$
by using $\Lambda^{-2}(f(u))_{xx}=\Lambda^{-2}f(u)-f(u)$. Calling $f(u)=u^2+(u^2)_x$ and observing that $2uu_x=(u^2)_x$, we get 
\begin{equation}\label{4.0.1}
u_t+(1-\partial_x)u^2=\Lambda^{-2}f(u).
\end{equation}
Now, we will differentiate (\ref{4.0.1}) with respect to $x$, simplify and write
\begin{equation}\label{4.0.2}
u_{tx}+(\partial_x-\partial_x^2)u^2=\Lambda^{-2}f(u)-u^2.
\end{equation}
We continue this process and get the following equations:
\begin{equation}\label{4.0.3}
u_{txx}+(\partial_x^2-\partial_x^3)u^2=\Lambda^{-2}f(u)-f(u),
\end{equation}
\begin{equation}\label{4.0.4}
u_{txxx}+(\partial_x^3-\partial_x^4)u^2=\partial_x\Lambda^{-2}f(u)-\partial_x f(u),
\end{equation}

\begin{equation}\label{4.0.5}
u_{txxxx}+(\partial_x^4-\partial_x^5)u^2=\Lambda^{-2}f(u)-f(u)-\partial_x^2 f(u).
\end{equation}

Moreover, we will multiply (\ref{4.0.1}) by $u$, (\ref{4.0.2}) by $u_x$, (\ref{4.0.3}) by $u_{xx}$, (\ref{4.0.4}) by $u_{xxx}$, (\ref{4.0.5}) by $u_{xxxx}$ and integrate all over $\s$.

Let $I(u)=\int_{\s}(u^2+u_x^2+u_{xx}^2+u_{xxx}^2+u_{xxxx}^2) dx$. Therefore, summing up equations (\ref{4.0.1})-(\ref{4.0.5}) we obtain
\begin{eqnarray*}
    &&\frac{1}{2}\frac{d}{dt}I(u)+\int_{\s}(u(1-\partial_x)u^2+u_x (\partial_x-\partial_x^2)u^2+u_{xx}(\partial_x^2-\partial_x^3)u^2\\
    &&+u_{xxx}(\partial_x^3-\partial_x^4)u^2+u_{xxxx}(\partial_x^4-\partial_x^5)u^2)dx\\
    &&=\int_{\s} (u\Lambda^{-2}f(u)+u_x\Lambda^{-2}f(u)-u^2u_x+u_{xx}(\Lambda^{-2}f(u)-f(u))\\
    &&+u_{xxx}(\partial_x\Lambda^{-2}f(u)-\partial_x f(u))+u_{xxxx}(\Lambda^{-2}f(u)-f(u)-\partial_x^2f(u))dx.
\end{eqnarray*}
Our main aim is to obtain $I(u)$, which is equivalent to $H^4$ norm of $u$, within the equation so that Gronwall's inequality is applicable and we get an upper bound valid for all time. That bound will imply the global existence of solution. 

After integration by parts, we can rewrite the equality in the following form:
\begin{eqnarray*}
    &&\frac{1}{2}\frac{d}{dt}I(u)+\int_{\s}(u^3+2(u^2)_x u_{xx}+(u^2)_{xx} u_{xxx}+2(u^2)_{xxx}u_{xxxx}\\
    &&+(u^2)_{xxxx} u_{xxxx}+(u^2)_{xxxx} u_{xxxxx}+u^2u_{xxxx})dx=\int_{\s}u(\Lambda^{-2}f(u))dx.
\end{eqnarray*}

Since 
\begin{eqnarray*}
    &&(u^2)_x=2uu_x,\quad (u^2)_{xx}= 2u_x^2+2uu_{xx},\quad (u^2)_{xxx}=6u_xu_{xx}+2uu_{xxx},\\
    &&(u^2)_{xxxx}=6u_{xx}^2+8u_xu_{xxx}+2uu_{xxxx},\\
    &&(u^2)_{xxxxx}=20u_{xx}u_{xxx}+10u_{x}u_{xxxx}+2uu_{xxxxx},
\end{eqnarray*}
and integrating by parts once more, the integral becomes
\begin{eqnarray}
    &&\frac{1}{2}\frac{d}{dt}I(u)+\int_{\s}(u^3-2(u^2)_{xx} u_{x}-(u^2)_{xxx} u_{xx}-2(u^2)_{xxxx}u_{xxx}\nonumber\\
    &&-(u^2)_{xxxxx} u_{xxx}-(u^2)_{xxxxx} u_{xxxx}+(u^2)_{xx}u_{xx})dx\nonumber\\
    &=&\frac{1}{2}\frac{d}{dt}I(u)+\int_{\s}(u^3-4u_x^3-4uu_xu_{xx}-4uu_{xx}^2-2uu_{xx}u_{xxx}\nonumber\\
    &&-16u_xu_{xxx}^2-4uu_{xxx}u_{xxxx}-20u_{xx}u_{xxx}^2\nonumber\\
    &&-10u_xu_{xxx}u_{xxxx}-2uu_{xxx}u_{xxxxx}-20u_{xx}u_{xxx}u_{xxxx}-10u_xu_{xxxx}^2\nonumber\\
    &&-2uu_{xxxx}u_{xxxxx})dx=\int_{\s}u(\Lambda^{-2}f(u))dx\label{4.0.6}.
\end{eqnarray}

By Lemma \ref{lem4.1}, we have that $\|u_x\|_{\infty}$ is bounded. Moreover, we can rewrite (\ref{1.0.1}) as
$$u_t-u_{txx}=4uu_x+2u_x^2+2uu_{xx}-6u_xu_{xx}-2uu_{xxx}.$$

Multiplying equation above by $u$, noting that
\begin{eqnarray*}
    2uu_xu_{xx}&=&\partial_x(uu_x^2)-u_x^3, \quad
    u^2u_{xxx}=\partial_x(u^2u_{xx}-uu_x^2)+u_x^3,\\
     u^2u_{xx}&=& \partial_x(u^2u_{x})-2uu_x^2,
\end{eqnarray*}
integrating over $\s$ and using the identities above, we obtain
\begin{eqnarray*}
   \frac{1}{2}\frac{d}{dt}\int_\s (u^2+u_x^2)dx=\int_\s (4u_x^3-2uu_x^2)dx\leq 5\|u_x\|_{\infty} \int_\s (u^2+u_x^2)dx,
\end{eqnarray*}
which implies 
$$\|u\|_1^2\leq \|u_0\|_{1}^2 e^{10\int_0^t\|u_x\|_{\infty}d\tau}
   \leq  \|u_0\|_{1}^2 e^{At}=C_0^2,$$
for some optimal constant $A>0$ since $\|u_x\|_{\infty}$ norm is bounded. This estimate is valid at any finite time, therefore is a global bound for $H^1$ norm of $u$. The reason we provide this inequality is to verify that  $\|u\|_{\infty}$ is also bounded, since
$\|u\|_{\infty} \leq  \|u\|_{1} \leq C_0$ by Sobolev embedding theorem. 

Moreover, we can show that $H^2$ norm of $u$ will be bounded in finite time:
Let $J(u)=\int_{\s}(u^2+u_x^2+u_{xx}^2)dx$. Following the arguments done above for (\ref{4.0.1})-(\ref{4.0.3}) and using integration by parts, 
\begin{eqnarray*}
     \frac{1}{2}\frac{d}{dt}J(u)&+&\int_{\s}(2u^3+(u^2)_{x} u_{x}+2(u^2)_{x} u_{xx}+(u^2)_{xx}u_{xx}+(u^2)_{xx}u_{xxx})dx\\
     &\leq& \int_{\s}2u(\Lambda^{-2}f(u))dx.
\end{eqnarray*}
Since $\|u\|_\infty$ and $\|u_x\|_\infty$ are bounded,
\begin{eqnarray*}
     \frac{1}{2}\frac{d}{dt}J(u)&\leq&2\max{(\|u\|_{\infty}, \|u_x\|_{\infty})}J(u)+2\max{(\|u\|_\infty)}\int_{\s}(\Lambda^{-2}f(u))dx.
\end{eqnarray*}
As it was given in (\ref{1.0.6}), $\Lambda^{-2}f=g*f$. Now, we need to estimate the $L^\infty-$norm of the integrand of the second term in order to obtain a differential inequality and apply Gronwall's inequality. For this purpose, we first provide the following two estimates:
\begin{eqnarray*}
    \|g\|_{2}\leq \frac{1}{2}(\frac{e^2+2e-1}{e^2-2e+1})^{1/2}:=n_2, \quad
     \|g\|_{\infty}\leq \frac{1}{2}(\frac{e+1}{e-1}):=n_{\infty}.
\end{eqnarray*}
Hence, 
\begin{eqnarray*}
    &&\|g*u^2\|_{\infty}\leq  \|g\|_{\infty} \|u^2\|_{1}\leq \|g\|_{\infty}\|u\|_{2}^2\leq n_{\infty} C_0,\\
    &&\|g*(u^2)_x\|_{\infty}\leq\|g\|_{2}\|(u^2)_x\|_{2}\leq \|g\|_{2} \|u^2\|_{1}\leq n_2 C_0^2.
\end{eqnarray*}
These estimates, together with Gronwall's inequality, provide the boundedness of $J(u)$ which is equivalent to $H^2$ norm. Therefore, we will be able to give an upper bound for $\|u_{x}\|_\infty$ norm as well since $\|u_x\|_{\infty}\leq \|u_x\|_1<\infty.$

Proving that $H^3$ norm is bounded in finite time will be the last issue to conclude the proof of Theorem \ref{thm1.3}:\\
Let $K(u)=\int_{\s}(u^2+u_x^2+u_{xx}^2++u_{xxx}^2)dx$. Like we did for $J(u)$, we can evaluate the following inequality:
\begin{eqnarray*}
     \frac{1}{2}\frac{d}{dt}K(u)&+&\int_{\s}(2u^3+(u^2)_{x} u_{x}+2(u^2)_{x} u_{xx}+(u^2)_{xx}u_{xx}\\
     &+&(u^2)_{xx}u_{xxx}+(u^2)_{xxx}u_{xxx}+(u^2)_{xxx}u_{xxxx})dx\\
     &\leq& \int_{\s}(2u(\Lambda^{-2}f(u))+u_x(\Lambda^{-2}f(u))+u_{xxx}\partial_x(\Lambda^{-2}f(u)))dx,
\end{eqnarray*}
and 
\begin{eqnarray*}
    \frac{1}{2}\frac{d}{dt}K(u)&\leq& 2\max{(\|u\|_{\infty}, \|u_x\|_{\infty}})K(u)+\int_{\s}(2u(\Lambda^{-2}(u^2))+2u_x(\Lambda^{-2}(u^2)_x))dx\\
    &\leq& 2\max{(\|u\|_{\infty}, \|u_x\|_{\infty}})K(u)+2\max{(\|u\|_{\infty}, \|u_{x}\|_\infty})\int_{\s}(\Lambda^{-2}f(u))dx.
\end{eqnarray*}

Similar arguments reveal that $K(u)$ is bounded in finite time and hence, $\|u_{xx}\|_\infty\leq \|u_{xx}\|_1<\infty.$

Recalling the equality (\ref{4.0.6}), and Theorem \ref{thm4.1} which guarantees that $u>0$, we evaluate
\begin{eqnarray*}
\frac{1}{2}\frac{d}{dt}I(u)&=&-(\int_{\s}(u^3-4u_x^3-4uu_xu_{xx}-4uu_{xx}^2-2uu_{xx}u_{xxx}\nonumber\\
    &&-16u_xu_{xxx}^2-4uu_{xxx}u_{xxxx}-20u_{xx}u_{xxx}^2\nonumber\\
    &&-10u_xu_{xxx}u_{xxxx}-2uu_{xxx}u_{xxxxx}-20u_{xx}u_{xxx}u_{xxxx}-10u_xu_{xxxx}^2\nonumber\\
    &&-2uu_{xxxx}u_{xxxxx})dx)+\int_{\s}u(\Lambda^{-2}f(u))dx\\
    &\leq&-(\int_{\s}(-u^3-4u_x^3-4uu_xu_{xx}-4uu_{xx}^2-2uu_{xx}u_{xxx}\nonumber\\
    &&-16u_xu_{xxx}^2-4uu_{xxx}u_{xxxx}-20u_{xx}u_{xxx}^2\nonumber\\
    &&-10u_xu_{xxx}u_{xxxx}-2uu_{xxx}u_{xxxxx}-20u_{xx}u_{xxx}u_{xxxx}-10u_xu_{xxxx}^2\nonumber\\
    &&-2uu_{xxxx}u_{xxxxx})dx)+\int_{\s}u(\Lambda^{-2}f(u))dx\\
    &\leq&\max{(\|u\|_{\infty}, \|u_x\|_{\infty}, \|u_{xx}\|_{\infty})}I(u)+\int_{\s}u(\Lambda^{-2}f(u))dx\\
    &\leq& \max{(\|u\|_{\infty}, \|u_x\|_{\infty}, \|u_{xx}\|_{\infty})}I(u)+\max{(\|u\|_{\infty})\int_{\s}}(\Lambda^{-2}f(u))dx.
\end{eqnarray*}

Therefore,
\begin{eqnarray*}
    \frac{1}{2}\frac{d}{dt}I(u)&\leq& K_2I(u)+K_3
\end{eqnarray*}
for some optimal constants $K_2$, $K_3$. Gronwall's inequality implies
$I(u)\leq [I(0)+K_3t]e^{K_2t}$, which is valid for any finite time $0<t\leq T$. Since we find an upper bound for $\|u\|_4$, this completes the proof of Theorem \ref{thm1.3}.

\section{Proof of theorem \ref{thm1.1}}\label{sec5}

The proof of theorem \ref{thm1.1} is divided in three major parts, namely,
\begin{itemize}
\item[{\bf P1}] Existence of $C^1$ periodic one-forms $\omega_1,\omega_2$ and $\omega_3$ satisfying \eqref{1.0.4};
\item[{\bf P2}] Existence of a domain $V$, depending on the initial datum, containing open sets endowed with a PSS structure;
\item[{\bf P3}] Existence of local connection forms $\omega_{13},\omega_{23}$.
\end{itemize}

{\bf $\bullet$ Existence of $C^1$ periodic one-forms $\omega_1$, $\omega_2$ and $\omega_3$.}

Example \ref{example2.1} exhibits two triads of one forms \eqref{2.3.6} satisfying the condition \eqref{1.0.4}. 

For solutions $u$ emanating from an initial datum $u_0\in H^4(\s)$, with $u_0-u_0''>0$, theorem \ref{thm1.2} implies that $u\in C(H^{4}(\mathbb{S}),[0,T))\cap C^1(H^{3}(\mathbb{S}),[0,T))$, whereas Theorem \ref{thm1.3} informs us that $u$ is defined on $U=\R\times(0,\infty)$. Moreover, $u_t(\cdot,t)\in H^3(\s)\subseteq C^2_{\text{per}}(\R)$ and $u(\cdot,t)\in H^4(\s)\subseteq C^3_{\text{per}}(\R)$ in view of the Sobolev Lemma (see lemma \ref{lemma2.3}). Therefore, $u\in C^{3,1}(\R)$ and then $f_{ij}\in C^1(\R)$ and is periodic in the variable $x$, where $f_{ij}$ are the coefficients  of the forms given in \eqref{2.3.6}.

={\bf $\bullet$ Existence of a domain $V$, depending on the initial datum, containing open sets endowed with a PSS structure};

From example \ref{example2.2}, a non-generic solution $u$ of \eqref{1.0.1} can only be periodic if it is constant. Since the initial datum satisfies the condition $u_0-u_0''>0$, by theorem \ref{thm4.1} we know that $u>0$ and it cannot be constant on $U$.

Let us then suppose the existence of an open set $\Omega\subseteq U$ for which $u\big|_\Omega=k$, where $k>0$ is a real number. Without loss of generality, we may assume that $\Omega\subseteq(0,1)\times(0,\infty)=:U_1$ in view of the periodicity of $u$ with respect to $x$. For some $p\in U_1\setminus\Omega$ and $\epsilon>0$, we have $\nabla u(p)=(u_x(p),u_t(p))\neq(0,0)$ and $u\big|_{B_\epsilon(p)}$ is non-constant, where $B_\epsilon(p)$ denotes the disc of centre $p$ and radius $\epsilon$. As a result, $U_1$ has at least one connected component $V$ (and $U$ as well), with $B_\epsilon(p)\subseteq V$, endowed with a PSS structure determined by the forms $\omega_1$ and $\omega_2$.

{\bf $\bullet$ Existence of connection forms defined everywhere $\omega_1\wedge\omega_2\neq0$.}

Henceforth we assume that the open sets under consideration are those that $\nabla u\neq(0,0)$ everywhere. Let us denote these sets generically by $V$.

Our proof is based on, and follows, that made by Castro Silva and Kamran \cite[Proposition 3.7]{tarcisio}.

We have two possible choices for the form $\omega_2$. For this reason, fix one of them and consider the frame $\{\omega_1,\omega_2\}$. Let $a$, $b$ and $c$ functions such that
\bb\label{5.0.4}
\omega_{13}=a\omega_1+b\omega_2,\quad \omega_{23}=b\omega_1+c\omega_2.
\ee

Our task is to find functions locally defined on any open set of $U$ for which \eqref{5.0.4} and the Levi-Civita connection form $\omega_3$ given in \eqref{2.3.6} satisfy \eqref{2.3.1}.

In \cite[Proposition 3.7]{tarcisio} it was shown that the connection forms \eqref{5.0.4} for an equation of the type
\bb\label{5.0.5}
u_t-u_{txx}=\lambda uu_{xxx}+G(u,u_x,u_{xx}),
\ee
satisfy a certain set of differential equations, see \cite[Theorem 2.4]{tarcisio} and also \cite[Theorem 3.4]{keti2015}. The function $G$ has somewhat a specific dependence on its arguments, and also some parameters $\mu$, $m_1$ and $m_2$. In \cite[Theorem 1]{freire-tito-sam} it was shown that \eqref{1.0.1} is a PSS equation, and one of the steps for that demonstration was just to show that it falls in the class considered in \cite[Theorem 3.4]{keti2015}. In particular, the mentioned parameters are $\mu\in\R$, $m_1\in\{-2,1\}$ and $m_2=0$.

Therefore, in view of \cite[Equation (7)]{freire-tito-sam} and \cite[Theorem 3.4]{keti2015}, we fall either into \cite[Proposition 3.7, case ii.)]{tarcisio} or \cite[Proposition 3.7, case iii.)]{tarcisio}, depending on whether $\mu=0$ or $\mu\neq0$.

According to \cite[Equation (181)]{tarcisio}, the functions $a$, $b$ and $c$ in \eqref{5.0.4} take the form (recall that $m_2=0$) $a=\phi_1(z)$, $b=\phi_2(z)$, and $c=\phi_3(z)$, $z=m_1x$, for some real valued and smooth functions $\phi_1$, $\phi_2$ and $\phi_3$ to be determined, satisfying the condition
\bb\label{5.0.6}
\phi_1\phi_3\neq0.
\ee

The Codazzi-Mainardi equations give (see \cite[Equations (182)--(183)]{tarcisio})
\bb\label{5.0.7}
\phi_1'+\mu\phi_2'-\phi_1-2\mu\phi_2+\phi_3=0
\ee
and
\bb\label{5.0.8}
\phi_2'+\mu\phi_3'+\mu\phi_1-2\phi_2-\mu\phi_3=0,
\ee
whereas the Gauss equation reads
\bb\label{5.0.9}
\phi_1\phi_3-\phi_2^2=-1.
\ee

Equation \eqref{5.0.9}, jointly with condition \eqref{5.0.6}, imply that $b\neq0$ everywhere. 

From \eqref{5.0.7} we obtain $\phi_3$ in terms of $\phi_1$, $\phi_2$ and their first derivatives. Substituting the result into \eqref{5.0.8} and integrating once, we obtain (see \cite[Equation (184)]{tarcisio})
\bb\label{5.0.10}
\mu\phi_1'=(1+\mu^2)\phi_2-\mu^2\phi_2'-\be e^{2z},
\ee
where $\be\in\R$ is a constant.

We now divide our proof in two different cases.

{\bf Case $\mu=0$.} From \eqref{5.0.10} we conclude that 
$\phi_2=\be e^{2z},$
whereas \eqref{5.0.7} gives
\bb\label{5.0.11}
\phi_3=\phi_1-\phi_1'.
\ee

Substituting $\phi_2$ and $\phi_3$ into the Gauss equation \eqref{5.0.9} we obtain the following Abel differential equation of the second kind
$$\phi_1\phi_1'=\phi_1^2-\be^2 e^{4z}+1.$$

Under the change $\phi_1=e^zw$, where $w$ is another function of $z$, we obtain the following simpler ODE 
$$ww'=e^{-2z}-\be^2 e^ {2z},$$
that, after solving, substituting back the result for $\phi_1$, and proceeding some manipulation, gives
\bb\label{5.0.12}
\phi_1(z)=\pm e^z\sqrt{\gamma-\be^2 e^{2z}-e^{-2z}}=\pm\sqrt{\gamma e^{2z}-1 -\be^2 e^{4z}}.
\ee

Substituting \eqref{5.0.12} into \eqref{5.0.11} and going back to the original functions $a$, $b$ and $c$, we obtain
\bb\label{5.0.13}
\ba{lcl}
a(x,t)&=&\ds{\pm \sqrt{\gamma e^{2m_1x}-\be^2e^{4m_1x}-1}},\quad b(x,t)=\ds{\be e^{2m_1x}},\\
\\
c(x,t)&=&\pm\ds{\f{\be^2 e^{2m_1x}-1}{ \sqrt{\gamma e^{2m_1x}-\be^2 e^{4m_1x}-1}}}.
\ea
\ee
]{\bf Case $\mu\neq0$.} From \eqref{5.0.10} we can write $\phi_1'$ in terms of $\phi_2$ and $\phi_2'$. Substituting it into \eqref{5.0.9} we obtain
\bb\label{5.0.14}
\phi_3=\phi_1+\phi,\quad
\phi=\ds{\f{\mu^2-1}{\mu}\phi_2-\f{\be}{\mu}e^{2z}.}
\ee

Substituting \eqref{5.0.14} into the Gauss equation \eqref{5.0.9} we obtain
$$\phi_1^2+\phi\phi_1-\phi_2^2=-1,$$
which, after solved for $\phi_1$, yields
\bb\label{5.0.15}
\phi_1=\f{-\phi\pm\sqrt{\Delta}}{2},\quad \Delta=\phi^2-4(1-\phi_2^2),
\ee
which are well defined as long as $\Delta\geq0$.

Substituting \eqref{5.0.15} into \eqref{5.0.14} and the result into \eqref{5.0.10} we obtain the following ODE for $\phi_2$
\bb\label{5.0.16}
[(1+\mu^2)\sqrt{\Delta}\pm(\mu^2-1)\phi\pm4\mu \phi_2]\phi_2' -2(1+\mu^2)\sqrt{\Delta}\phi_2\mp 2\be e^{2z}\phi=0.
\ee

It was shown in \cite[page 36]{tarcisio} that the coefficient of $\phi_2'$ cannot vanish, otherwise we would conclude that $\phi_1=\phi_3$ and $\phi_2=0$, which contradicts \eqref{5.0.9}. By continuity, such a coefficient does not change its sign. Without loss of generality, we may assume it to be positive and the ODE above takes the form $b'=g(z,b)$.

Given a point $p=(x_0,t_0)\in V$, arguing exactly as \cite[page 37]{tarcisio} we conclude that the ODE to $b$ subject to $b(x_0/m_1)=t_0$ has a unique (local) solution, that guarantee the local existence of $b$ in a neighborhood of each point of $V$.\hfill$\square$

\begin{remark}
    The uniqueness of the forms \eqref{2.3.6} follows from \cite[Theorem 3.4]{keti2015}, see also \cite[page 760]{freire-tito-sam}, and the fact that the solution $u$ of \eqref{1.0.1} is unique in view of theorem \ref{thm1.2}. 
\end{remark}

\section{Concluding remarks} \label{sec6}

In this paper we studied an equation whose solutions define metrics for a PSS from the point of view of geometric analysis. More precisely, we used tools of semi-group theory to establish well-posedness of solutions and then study the corresponding surface qualitatively.

From the point of view of analysis, our Theorem \ref{thm1.2} ensures well-posedness of solutions, whereas our Theorem \ref{thm1.3} ensures enough regularity of the solution in order to ensure geometric relevance in the one-forms given in Theorem \ref{thm1.1}. In particular, this last theorem can be seen as a sort of existence and uniqueness theorem for periodic surfaces emanating from a given initial datum, which can be associated to a certain curve in the three-dimensional space.

\section*{Acknowledgements}

N. D. Mutluba\c{s} is supported by the Turkish Academy of Sciences within the framework of the Outstanding Young Scientists Awards Program (T\"{U}BA-GEBIP-2022). I. L. Freire is thankful to Enrique Reyes, Marcio Fabiano da Silva and Stefano Nardulli for stimulating discussions and support. He is also grateful to CNPq (grant nº 310074/2021-5) and FAPESP (grants nº 2020/02055-0 and 2022/00163-6) for financial support. He is also grateful to the Department of Mathematical Sciences for warm hospitality and support during the development of this work.

\end{document}